\documentclass[review,3p,10pt]{elsarticle}
\biboptions{sort&compress}

\usepackage{amsmath,amssymb,amsthm}
\usepackage{bm,color}
\usepackage{stmaryrd}
\usepackage[T1]{fontenc}
\usepackage{float, caption, subcaption}
\usepackage{booktabs}
\usepackage{multirow}
\usepackage{subcaption}
\usepackage[bookmarks=true,
  bookmarksnumbered=true,
  bookmarksopen=true,
  colorlinks,
  linkcolor=black]{hyperref}

\allowdisplaybreaks
\newtheorem{Def}{Definition}[section]
\newtheorem{theorem}{Theorem}[section]
\newtheorem{lemma}[theorem]{Lemma}

\newtheorem{remark}{Remark}[section]

\makeatletter
\@addtoreset{equation}{section}
\@addtoreset{equation}{subsection}
\renewcommand{\theequation}{%
  \thesection
  \ifnum\value{subsection}>0 .\arabic{subsection}\fi
  .\arabic{equation}}
\@addtoreset{remark}{subsection}
\renewcommand{\theremark}{%
  \thesection
  \ifnum\value{subsection}>0 .\arabic{subsection}\fi
  .\arabic{remark}}
\makeatother
\numberwithin{figure}{subsection}
\numberwithin{table}{subsection}

\DeclareMathAlphabet\mathbfcal{OMS}{cmsy}{b}{n}

\newcommand\pd[2]{\frac{\partial {#1}}{\partial {#2}}}

\journal{{\em JCP’s Special Issue in honor of Peter Lax}}
\usepackage{pifont,xcolor}

\begin{document}
\begin{frontmatter}
  \title{E Scheme and Flux-Limiter Scheme,  Revisited}
  \author{Huazhong Tang \corref{cor1}}
  \ead{hztang@math.pku.edu.cn,hztang@pku.edu.cn}
 \cortext[cor1]{Corresponding author.}  
\address{Center for Applied Physics and Technology, HEDPS and LMAM, 
School of Mathematical Sciences, Peking University, Beijing 100871, P.R. China}
\date{\today}
  \begin{abstract}
This paper  revisits
 the {\em E scheme} of Osher \cite{Osher-SINUM1984}
and the {\em flux-limiter scheme} of Sweby for quasi-linear hyperbolic conservation laws \cite{Sweby-SINUM1984}.
 Part of existing results will be re-understood and some new results will be presented. 
For a scalar conservation law,  except for the
conservative monotone schemes, the E scheme  is a type
of numerical methods that satisfy the discrete entropy condition for any convex entropy, but  numerical entropy flux is not unique.
 {Two-point monotone flux is E flux, but conversely it may not necessarily
be correct. Moreover,  }
 multi-point (three or more points) E flux may not necessarily be monotone flux, and multi-point monotone flux may not necessarily be E flux.
  Sweby's  flux-limiter scheme for the quasi-linear conservation laws was built on the E flux-based splitting  $f_{j+1}-f_j=f_{j+1} { -\hat{f}^{\text{\tiny E}}_{j+\frac12}+\hat{f}^{\text{\tiny E}}_{j+\frac12}}-f_j$ and the LW scheme. It may not be second-order accurate in both space and time.

\end{abstract}

\end{frontmatter}


\section{Introduction}

The conservation laws of gas dynamics, elastodynamics, electrodynamics and other branches of classical physics are typically expressed by hyperbolic partial differential equations (or systems).
In particular,  hyperbolic equation provides the proper mathematical setting for a host of wave phenomena.
 For quasi-linear hyperbolic conservation laws, 
  even if the initial value is smooth, the solution may be
discontinuous at finite time, so that  
the weak solutions in the sense of
distribution should be considered. 
Unfortunately, the weak solutions may not be unique, so that 
the admissible condition such as the entropy condition is require to select
the unique admissible physical solution.

For numerical methods for hyperbolic conservation laws, the E scheme  \cite{Osher-SINUM1984} for a scalar conservation law is a type of numerical methods that satisfy the discrete entropy condition for any convex entropy, except for the conservative monotone schemes,
and it has been observed that {low-order schemes} are usually stable but quite dissipative in nature around the points of discontinuity. On the other hand, 
  {higher-order numerical schemes} are unstable in nature and show oscillations in the vicinity of discontinuity.
  The  {idea of the flux-limiter scheme}  \cite{Sweby-SINUM1984} is to tune the numerical flux of high order and low order scheme using the  {flux limiter} in such a way that the resulting scheme gives a high order accuracy in the smooth region of flow and sticks with low order of accuracy in the vicinity of sock waves and other discontinuities.  
The  flux-limiter technology is still frequently applied in the construction of high-order accurate bound/positivity-preserving schemes \cite{Wu-Tang-JCP2015,Xu-Zhang2017}.

  {The purpose of this paper is to revisit} 
 the {E scheme} of Osher 
and the {flux-limiter scheme} of Sweby for the quasilinear hyperbolic conservation laws. 
  Part of existing results will be re-understood and some new results will be presented. 
The remainder of this paper is organized as follows. 
Section \ref{section02}
introduces the definition of semi-discrete E scheme as well as E flux,
and discusses their properties.
Section \ref{section03} visits the two-level fully discrete E schemes and
their viscous coefficients and entropy inequalities.
 Section \ref{section04} revisits the {flux-limiter scheme} of Sweby
 for quasi-linear conservation law. 
Finally, our concluding remarks are given in
Section \ref{section05}.

\section{Semi-discrete E schemes}\label{section02}

Consider {one-dimensional scalar} {conservation law}
 \begin{align}\label{eq:HCL001}
u_t+f(u)_x=0,\ \ u,x\in{\mathbb R}, \ t>0,
\end{align}
and its semi-discrete { conservative} scheme
 \begin{align}\label{eq:scheme001}
\frac{\rm d}{{\rm d}t} u_j(t)=-\frac1h (\hat{f}_{j+\frac12}-\hat{f}_{j-\frac12}),
\ \ j\in{\mathbb Z}, \  t>0,
\end{align}
where $h$ denotes the step-size in space, $u_j(t)\approx u(x_j,t)$. 

\begin{remark}  
Numerical flux $\hat{f}_{j+\frac12}$ may depend on $2k$-variables, 
  $u_{j-k+1}$, $\cdots$,    $u_{j+k}$, $k\geq 1$, and
 is a { Lipschitz continuous} function of its arguments satisfying the { consistency} condition
 $ \hat{f}(u, u, \cdots,    u)= {f}(u)$.
 \end{remark}

\begin{Def}[E scheme] 
The semi-discrete scheme \eqref{eq:scheme001} is called  {\em E scheme}\cite{Osher-SINUM1984}, if
\begin{align}\label{eq:E flux001}
\mbox{sgn}(u_{j+1}-u_j) \big(\hat{f}_{j+\frac12}-f(u)\big)\leq 0,
\end{align}
for any $u$ between $u_j$ and $u_{j+1}$, $\forall j\in \mathbb Z$. 
 Denote the numerical flux in \eqref{eq:E flux001} as $\hat{f}_{j+\frac12}^{\text{\tiny E}}$.
 \end{Def}

\begin{remark}  \label{rem-2.2}
 The E schemes \eqref{eq:scheme001}--\eqref{eq:E flux001} may depend on { more than three points},
for example, { $\hat{f}^{\text{\tiny E}}(u_j,u_{j+1})  
-\operatorname{sgn}(u_{j+1}-u_j)  (g_{j+2}-g_{j-1})^2$}
satisfies \eqref{eq:E flux001}, for any { Lipschitz continuous} function $g(u)$.
  Some other possible  { 4-} and { 6}-point E fluxes satisfying \eqref{eq:E flux001}: 
 { $\hat{f}^{\text{\tiny E}}(u_j,u_{j+1})  
-(u_{j+1}-u_j)  (u_{j+2}-u_{j-1})^2$};
 { $\hat{f}^{\text{\tiny E}}(u_j,u_{j+1})  
-(u_{j+1}-u_j)  |u_{j+2}-u_{j-1}|$};
 { $\hat{f}^{\text{\tiny E}}(u_j,u_{j+1})  
-(u_{j+1}-u_j)  (f_{j+2}-f_{j-1})^2$};
 { $\hat{f}^{\text{\tiny E}}(u_j,u_{j+1})  
-(u_{j+1}-u_j)  |f_{j+2}-f_{j-1}|$};
 { $\hat{f}^{\text{\tiny E}}(u_j,u_{j+1})  
-(u_{j+1}-u_j)  (\hat{g}(u_{j+2},u_{j+3})-\hat{g}(u_{j-1},u_{j-2}) )^2$};
 { $\hat{f}^{\text{\tiny E}}(u_j,u_{j+1})  
-(u_{j+1}-u_j)  |\hat{g}(u_{j+2},u_{j+3})-\hat{g}(u_{j-1},u_{j-2})|$}.

\end{remark}   

\begin{remark}   The scheme \eqref{eq:scheme001} may be rewritten as
{\small  \begin{align*}
\frac{\rm d}{{\rm d}t} u_j(t)=-\frac1h \frac{\hat{f}^{\text{\tiny E}}_{j+\frac12}-f_j }{u_{j+1}-u_j} (u_{j+1}-u_j)
-\frac1h \frac{f_j-\hat{f}^{\text{\tiny E}}_{j-\frac12}}{u_{j}-u_{j-1}} (u_{j}-u_{j-1})
=C_{j+\frac12} (u_{j+1}-u_j)
-D_{j-\frac12} (u_{j}-u_{j-1}),
\end{align*}}%
and \eqref{eq:E flux001} implies $C_{j+\frac12}\geq 0$, $C_{j-\frac12}\geq 0$ for all $j$,
so that  the E scheme \eqref{eq:scheme001}--\eqref{eq:E flux001} is TVD.
 Here $f_{j}=f(u_{j})$.
\end{remark}   

 The E flux condition \eqref{eq:E flux001} implies that
 \begin{itemize}
\item {if $u_{j+1}\geq u_j$},  then $\hat{f}_{j+\frac12}^{\text{\tiny E}}\leq f(u)$
 for  $\forall u\in[u_j,u_{j+1}]$, or {$\hat{f}_{j+\frac12}^{\text{\tiny E}}\leq \min_{u\in[u_j,u_{j+1}]} f(u)$};
 \item {if $u_{j+1}<u_j$},  then $\hat{f}^{\text{\tiny E}}_{j+\frac12}\geq f(u)$ for $\forall u\in[u_{j+1},u_{j}]$, 
  or {$\hat{f}^{\text{\tiny E}}_{j+\frac12}\geq \max_{u\in[u_{j+1},u_{j}]} f(u)$},
  \end{itemize}
so the {  E flux condition \eqref{eq:E flux001}} essentially requires that $\hat{f}^{\text{\tiny E}}$ does not ``cross over'' the range of the true $f(u)$, especially in nonlinear cases to prevent non-physical oscillations.

 \begin{lemma}
  { Two-point E flux $\hat{f}^{\text{\tiny E}}_{j+\frac12}$ is monotone flux}, i.e.  $\hat{f}^{\text{\tiny E}}(u_j, u_{j+1})=\hat{f}^{\text{\tiny Mon}}(u_j, u_{j+1})$, which is monotonically non-decreasing and  non-increasing
 w.r.t. $u_j$ and $u_{j+1}$, respectively.
 { Conversely, it may not necessarily be correct.}
 \end{lemma}
 \begin{proof}
 If $u_{j+1}\geq u_j$, then for $u \in [u_{j}, u_{j+1}]$, one has
\begin{align*}
 \hat{f}_{j+\frac12}^{\text{\tiny Mon}}{ -f(u)}
 =& \hat{f}^{\text{\tiny Mon}}(u_j, u_{j+1}) { -\hat{f}^{\text{\tiny Mon}}(u, u)}
 = \hat{f}^{\text{\tiny Mon}}(u_j, u_{j+1})    { -\hat{f}^{\text{\tiny Mon}}(u_j, u)+  \hat{f}^{\text{\tiny Mon}}(u_j, u)}               { -\hat{f}^{\text{\tiny Mon}}(u, u)}\leq 0,
 \end{align*}
 so that $\hat{f}_{j+\frac12}^{\text{\tiny Mon}}$  satisfies \eqref{eq:E flux001}.
  Similarly, if  $u_{j+1}< u_j$, then for $u \in [u_{j+1}, u_{j}]$, 
      $\hat{f}_{j+\frac12}^{\text{\tiny Mon}}$ also satisfies \eqref{eq:E flux001}.
 Conversely, { although}
$\operatorname{sgn}(u_{j+1}-u_j)(\hat{f}^{\text{\tiny E}}(u_j, u_{j+1})-f_j)=
\operatorname{sgn}(u_{j+1}-u_j)(\hat{f}^{\text{\tiny E}}(u_j, u_{j+1})-\hat{f}^{\text{\tiny E}}(u_j, u_{j}))\leq 0$, 
{ one cannot conclude that }  $\hat{f}^{\text{\tiny E}}(u_j, u_{j+1})$ is monotonically non-increasing w.r.t. the 2nd argument $u_{j+1}$;
although $\operatorname{sgn}(u_{j+1}-u_j)(\hat{f}^{\text{\tiny E}}(u_j, u_{j+1})-f_{j+1})=
\operatorname{sgn}(u_{j+1}-u_j)(\hat{f}^{\text{\tiny E}}(u_j, u_{j+1})-\hat{f}^{\text{\tiny E}}(u_{j+1}, u_{j+1}))\leq 0$,  one cannot conclude that   $\hat{f}^{\text{\tiny E}}(u_j, u_{j+1})$ is monotonically non-decreasing w.r.t. the 1st argument $u_{j}$.
 \end{proof}

Here are several well-known two-point E fluxes.
\begin{itemize}
  \item  {Godunov flux} \cite{LeVeque1992}
 $$\hat{f}_{j+\frac12}^{\text{\tiny God}}
 =\begin{cases}
 \min_{u\in[u_j,u_{j+1}]}  f(u), & u_{j+1}>u_j,\\
  \max_{u\in[u_{j+1},u_{j}]}  f(u), & u_{j+1}<u_j,\end{cases}
 $$ 
 is the upper { (resp. lower)} boundary of  the E flux   if $u_{j+1}\geq u_j$ { (resp. $u_{j+1}< u_j$)}. 
  \item {generalized Lax-Friedrichs flux}
 $$\hat{f}_{j+\frac12}^{\text{\tiny gLF}}
 =\frac12(f_j+f_{j+1})-\frac{q}{2}(u_{j+1}-u_j),\ \  \max_u\{|f'(u)|\}\leq q.
 $$
 
If $u_{j+1}\geq u_j$, then for $u \in [u_{j}, u_{j+1}]$, one has
\begin{align*}
 \hat{f}_{j+\frac12}^{\text{\tiny gLF}}{ -f(u)}
 =&\frac12\big(f_j{ f(u)}+f_{j+1}{ -f(u)}\big)-\frac{q}{2}(u_{j+1}{ -u+u}-u_j)
 \\
 =& -\frac12 \int_{u}^{u_{j+1}} [q-f'(\xi)]{\,\rm d}\xi
 -\frac12 \int^{u}_{u_{j}} [q+f'(\xi)]{\,\rm d}\xi\leq 0,
 \end{align*}
 so that $\hat{f}_{j+\frac12}^{\text{\tiny gLF}}$  satisfies \eqref{eq:E flux001}.
  Similarly, if  $u_{j+1}\leq u_j$, then for $u \in [u_{j+1}, u_{j}]$, 
      $\hat{f}_{j+\frac12}^{\text{\tiny gLF}}$ also satisfies \eqref{eq:E flux001}.
  \item {Engquist-Osher flux} \cite{E-O1980}
 \begin{align*}
 \hat{f}_{j+\frac12}^{\text{\tiny EO}}
 =& \frac12 (f_j+f_{j+1})-\frac12\int^{u_{j+1}}_{u_{j}} |f'(\xi)| {\,\rm d} \xi.
 %
 \end{align*}
 If $u_{j+1}\geq u_j$, then for $u \in [u_{j}, u_{j+1}]$, one has
  \begin{align*}
 \hat{f}_{j+\frac12}^{\text{\tiny EO}}{ -f(u)}
 =&\frac12\big(f_j{ -f(u)}+f_{j+1}{ -f(u)}\big)
 -\frac12\int^{u_{j+1}}_{{ u}} |f'(\xi)| {\,\rm d} \xi
  -\frac12\int^{{ u}}_{u_{j}} |f'(\xi)| {\,\rm d} \xi
 \\
= & -\frac12 \int_{u}^{u_{j+1}} [|f'(\xi)|-f'(\xi)]{\,\rm d}\xi
 -\frac12 \int^{u}_{u_{j}} [|f'(\xi)|+f'(\xi)]{\,\rm d}\xi\leq 0,
  \end{align*}
   so that $\hat{f}_{j+\frac12}^{\text{\tiny EO}}$  satisfies \eqref{eq:E flux001}.
 Similarly, if $u_{j+1}\leq u_j$, then for $u \in [u_{j+1}, u_{j}]$, 
    $\hat{f}_{j+\frac12}^{\text{\tiny EO}}$ also satisfies \eqref{eq:E flux001}.
 
\end{itemize}

 \begin{lemma}\label{lemma2.2}
Multi-point { (three or more points)} E flux may not necessarily be monotone flux, and multi-point monotone
flux may not necessarily be E flux.
\end{lemma}
\begin{proof}
Let us provide examples to illustrate this conclusion.
For example, consider { $\hat{f}^{\text{\tiny E}}(u_j,u_{j+1})  
-\operatorname{sgn}(u_{j+1}-u_j)  (g_{j+2}-g_{j-1})^2$}
satisfies \eqref{eq:E flux001} for any { Lipschitz continuous} function $g(u)$, given in Remark \ref{rem-2.2}. 
It  is not difficult to verify that corresponding scheme
{ \begin{align*}
u_{j}^{n+1}=&u_{j}^{n}-\sigma_n 
\big[  
 \hat{f}^{\text{\tiny E}}(u_j^n,u_{j+1}^n) 
 -\operatorname{sgn}(u_{j+1}-u_j)  (g_{j+2}-g_{j-1})^2
\\ &
-\hat{f}^{\text{\tiny E}}(u_{j-1}^n,u_{j}^n)  
+\operatorname{sgn}(u_{j}-u_{j-1})  (g_{j+1}-g_{j-2})^2
\big]
\\
=&u_{j}^{n}-\sigma_n 
\big[   \hat{f}^{\text{\tiny E}}(u_j^n,u_{j+1}^n) -\hat{f}^{\text{\tiny E}}(u_{j-1}^n,u_{j}^n)\big]
\\ & 
+\sigma_n \big[ 
\operatorname{sgn}(u_{j+1}-u_j)  (g_{j+2}-g_{j-1})^2
-\operatorname{sgn}(u_{j}-u_{j-1})  (g_{j+1}-g_{j-2})^2
\big],
\end{align*}}%
is not a monotone scheme in general for a generally { Lipschitz continuous} function $g(u)$.

Conversely, consider
{\small \begin{align}\label{eq:monotone}
u_{j}^{n+1}=&u_{j}^{n}-\frac14 \sigma_n 
\left( f^n_{j+2}-f_{j-2}^n \right)
+\frac14 \sigma_n q\left(u_{j+2}^n-2u_j^n+u_{j-2}^n\right)=:H(u_{j-2}^n,u_j^n,u_{j+2}^n)
\\
=&
u_{j}^{n}-\sigma_n 
\left[ \frac14 (f^n_{j+2}{ +f_{j+1}^n+f_j^n+f_{j-1}^n})
-\frac14 (f_{j-2}^n{ +f_{j-1}^n+f_j^n+f_{j+1}^n})
 \right]
+\frac14 \sigma_n q\left(u_{j+2}^n-2u_j^n+u_{j-2}^n\right).
\end{align}
}%
where $q\geq \max_u \{|f'(u)|\}$, $\sigma_n q\leq 2$.
The scheme \eqref{eq:monotone} is monotone, because
{\small \begin{align*}
\pd{H}{u_{j\pm 2}^n}=\frac14 \sigma_n (q\mp f'(u_{j-2}^n))\geq 0,\ \
\pd{H}{u_{j}^n}=1-\frac12\sigma_n q \geq0.
\end{align*}}%
It is obvious that the numerical flux of the scheme \eqref{eq:monotone} 
{\small $$
\hat{f}_{j+\frac12}=
\frac14 (f_{j+2}{ +f_{j+1}+f_j+f_{j-1}})
-\frac q4 (u_{j+2}^n-u_j^n),
$$}%
does not  satisfy the E flux condition \eqref{eq:E flux001}. 
%
\end{proof} 
 


\begin{theorem}\label{theorem001}
For any convex entropy pair $\{\eta(u),q(u)\}$ of HCL \eqref{eq:HCL001},
$\eta''>0$, $\eta' f'=q'$, 
the scheme \eqref{eq:scheme001}-\eqref{eq:E flux001} satisfies 
\begin{equation}\label{w1-E-scheme}
\frac{{\rm d}}{{\rm d}t} \eta( u_j(t)) +\frac{1}{h}(\hat{q}_{j+\frac12}-\hat{q}_{j-\frac12})\leq 0.
\end{equation}
 \end{theorem}

\begin{proof}
For any convex entropy pair $\{\eta(u),q(u)\}$ of HCL \eqref{eq:HCL001},
$\eta''>0$, $\eta' f'=q'$,  
multiplying  \eqref{eq:scheme001} by $\eta'_j:=\eta'(u_j)$
 gives
 \begin{align}\nonumber
0=&\eta'_j \frac{{\rm d}}{{\rm d}t} u_j(t)+\frac{1}{h}(\eta'_j\hat{f}_{j+\frac12}-\eta'_j\hat{f}_{j-\frac12})
\\  \nonumber
=&\frac{{\rm d}}{{\rm d}t} \eta(u_j)+\frac{1}{h}(\eta'_j\hat{f}_{j+\frac12}
{ - \eta'_{j+1}\hat{f}_{j+\frac12} +}  \underbrace{{ \eta'_{j+1}\hat{f}_{j+\frac12}}-\eta'_j\hat{f}_{j-\frac12}  }_{\text{\tiny divergence form}})
\\ \nonumber
=&\frac{{\rm d}}{{\rm d}t} \eta(u_j)+\frac{1}{h}( \eta'_{j+1}\hat{f}_{j+\frac12}-\eta'_j\hat{f}_{j-\frac12})
+\frac{1}{h}(\eta'_j- \eta'_{j+1})\hat{f}_{j+\frac12}
\\ \label{eq-E-scheme001}
=& \frac{{\rm d}}{{\rm d}t} \eta(u_j)+\frac{1}{h}( \eta'_{j+1}\hat{f}_{j+\frac12}-\eta'_j\hat{f}_{j-\frac12})
-\frac{1}{h}\int^{u_{j+1}}_{u_j}  \eta'' (\xi)  \hat{f}_{j+\frac12} {\rm d}\xi.
\end{align}
Note that the purpose of adding and subtracting terms on the right side of the second equal sign is to match the discrete divergence form. 

On the other hand, because for any convex entropy$\eta$,  
the  { E flux} condition \eqref{eq:E flux001} implies
  \begin{align*}
 \int^{u_{j+1}}_{u_j}  \eta'' (\xi)  \big(\hat{f}_{j+\frac12} -f(\xi)\big) {\,\rm d}\xi \leq 0,
  \end{align*}
  so that
   \begin{align}\nonumber
 \int^{u_{j+1}}_{u_j} & \eta'' (\xi) \hat{f}_{j+\frac12}   {\,\rm d}\xi 
 \leq 
  \int^{u_{j+1}}_{u_j}  \eta'' (\xi) f(\xi) {\,\rm d}\xi 
  =  \int^{u_{j+1}}_{u_j} \big[ (\eta' (\xi) f(\xi))'-\eta'f'\big]  {\,\rm d}\xi 
  \\ \nonumber
   & =  \int^{u_{j+1}}_{u_j} \big[ (\eta' (\xi) f(\xi))'-q'\big]  {\,\rm d}\xi 
  \\
&=  \eta' ({u_{j+1}}) f({u_{j+1}})-\eta' ({u_{j}}) f({u_{j}}) 
-q({u_{j+1}})+q({u_{j}}).  \label{eq-E-scheme002}
  \end{align}
 The right hand side of \eqref{eq-E-scheme002} is in a discrete divergence form.
 
Combining \eqref{eq-E-scheme001} with \eqref{eq-E-scheme002}
gives 
 \begin{align*}
0\geq &\frac{{\rm d}}{{\rm d}t} \eta(u_j)+\frac{1}{h}( \eta'_{j+1}\hat{f}_{j+\frac12}
 -\eta'_j\hat{f}_{j-\frac12})
\\ & -\frac{1}{h} \left(\eta' ({u_{j+1}}) f({u_{j+1}})-\eta' ({u_{j}}) f({u_{j}}) 
-q({u_{j+1}})+q({u_{j}})\right).
  \end{align*}
  This is the entropy inequality to be proven, where the numerical entropy flux
  $$\hat{q}_{j+\frac12}=\eta'_{j+1}\hat{f}_{j+\frac12}-\eta' ({u_{j+1}}) f({u_{j+1}})+q({u_{j+1}}),
  $$
  is consistent with the entropy flux $q(u)$.
\end{proof}

\begin{remark}
From the proof of { Theorem \ref{theorem001}} that the scheme satisfies the entropy inequality, it can be seen that 
{ the numerical entropy flux in the discrete entropy inequality is not unique}.
  The form of those added and subtracted terms in \eqref{eq-E-scheme001} is not unique.
For example, it can also be added or subtracted $\eta'_{j-1}\hat{f}_{j-\frac12}$, 
or $a\eta'_{j+1}\hat{f}_{j+\frac12}+b\eta'_{j-1}\hat{f}_{j-\frac12}$, 
where both $a$ and $b$ are  arbitrarily positive constants, satisfying { $a+b=1$}.
 \begin{align*}
&\eta'_j\hat{f}_{j+\frac12}-\eta'_j\hat{f}_{j-\frac12}
=(a+b) \eta'_j\hat{f}_{j+\frac12} 
-(a+b) \eta'_j\hat{f}_{j-\frac12} 
\\
&=a\eta'_j\hat{f}_{j+\frac12}
{+b\eta'_j\hat{f}_{j+\frac12}}
{-a\eta'_j\hat{f}_{j-\frac12}}
-b\eta'_j\hat{f}_{j-\frac12}
\\ &\qquad \qquad -a\eta'_{j+1}\hat{f}_{j+\frac12}   
{-b\eta'_{j-1}\hat{f}_{j-\frac12}}
{+a\eta'_{j+1}\hat{f}_{j+\frac12} }  
+b\eta'_{j-1}\hat{f}_{j-\frac12}
\\
&={b\eta'_j\hat{f}_{j+\frac12}}
{-b\eta'_{j-1}\hat{f}_{j-\frac12}}
{+a\eta'_{j+1}\hat{f}_{j+\frac12} } 
{-a\eta'_j\hat{f}_{j-\frac12}}
+a (\eta'_j-\eta'_{j+1})\hat{f}_{j+\frac12} 
+b (\eta'_{j-1}-\eta'_j)\hat{f}_{j-\frac12}.
\end{align*}
The four terms in the third line on the right side of the equal sign are addition and subtraction terms.
The first four  terms at the right hand sie  are in a discrete divergence form, 
while the first black term satisfies (the second is done similarly)
{\small \begin{align*}
& (\eta'_j-\eta'_{j+1})\hat{f}_{j+\frac12}
 =-\int_{u_j}^{u_{j+1}} \eta''(\xi) \hat{f}_{j+\frac12}{\,\rm d}\xi
 =-\int_{u_j}^{u_{j+1}} \eta''(\xi) (\hat{f}_{j+\frac12}-f(\xi)) {\,\rm d}\xi
 -\int_{u_j}^{u_{j+1}} \eta''(\xi)f(\xi){\,\rm d}\xi
 \\ &
 \geq  -\int_{u_j}^{u_{j+1}} \eta''(\xi)f(\xi){\,\rm d}\xi
 =-\int_{u_j}^{u_{j+1}}  [(\eta'(\xi)f(\xi))'- \eta'f'] {\,\rm d}\xi
 ={ [\eta'(\xi)f(\xi)- q]^{u_j}_{u_{j+1}}}.
 \end{align*}}%
  Thus, one has the following entropy inequality 
   \begin{align*}
  0=&\eta'_j \frac{{\rm d}}{{\rm d}t} u_j(t)+\frac{1}{h}(\eta'_j\hat{f}_{j+\frac12}-\eta'_j\hat{f}_{j-\frac12})
  \\
=&\eta'_j \frac{{\rm d}}{{\rm d}t} u_j(t)+\frac{1}{h}\Big[
{ b\eta'_j\hat{f}_{j+\frac12}}
{ -b\eta'_{j-1}\hat{f}_{j-\frac12}}
{ +a\eta'_{j+1}\hat{f}_{j+\frac12} } 
{ -a\eta'_j\hat{f}_{j-\frac12}}
\\ &
+a (\eta'_j-\eta'_{j+1})\hat{f}_{j+\frac12} 
+b (\eta'_{j-1}-\eta'_j)\hat{f}_{j-\frac12}\Big]
\\
\geq &
\eta'_j \frac{{\rm d}}{{\rm d}t} u_j(t)+\frac{1}{h}\Big[
{ b\eta'_j\hat{f}_{j+\frac12}}
{ -b\eta'_{j-1}\hat{f}_{j-\frac12}}
{ +a\eta'_{j+1}\hat{f}_{j+\frac12} } 
{ -a\eta'_j\hat{f}_{j-\frac12}}
\\ &
+a [\eta'(\xi)f(\xi)- q]^{u_j}_{u_{j+1}}
+b  [\eta'(\xi)f(\xi)- q]^{u_{j-1}}_{u_{j}}  \Big].
   \end{align*}
\end{remark}

\section{Fully discrete E schemes}\label{section03}

If applying  the { explicit Euler time discretiztion} to the semi-discrete E scheme \eqref{eq:scheme001}-\eqref{eq:E flux001}, 
then one has the fully discrete  (explicit) E scheme
\begin{align}\label{fully-discrete-E-scheme01}
u_j^{n+1}=u_j^n-\sigma_n(\hat{f}^{\text{\tiny E},n}_{j+\frac12}-\hat{f}^{\text{\tiny E},n}_{j-\frac12})=H(u_{j-k}^{n},\cdots,u_{j+k}^{n})
,\ \  \sigma_n=\frac{\tau_n}{h}.
\end{align}
The numerical flux may be rewritten as 
$$
\hat{f}^{\text{\tiny E}}_{j+\frac12}
=\frac12(f_j+f_{j+1})-\frac1{2\sigma_n}  {Q}^{\text{\tiny E}}_{j+\frac12} (u_{j+1}-u_j),
$$
where
$$
\frac1{\sigma_n}{Q}^{\text{\tiny E}}_{j+\frac12}({u_{j+1}-u_j})
=  (f_j+f_{j+1}-2\hat{f}^{\text{\tiny E}}_{j+\frac12}),
$$
or
$$
{Q}^{\text{\tiny E}}_{j+\frac12}=
{\sigma_n}   \frac{f_j+f_{j+1}-2\hat{f}^{\text{\tiny E}}_{j+\frac12} }  {u_{j+1}-u_j}.
$$
 The condition \eqref{eq:E flux001} for the E scheme implies 
 the viscous coefficient is non-negative, i.e.
 { ${Q}^{\text{\tiny E}}_{j+\frac12}\geq 0$}.
 
Furthermore, one can show that it is not less than the viscous coefficient
of Murman \cite{Murman1974}.
\begin{lemma}\label{lemma3.1}
\ ${Q}^{\text{\tiny E}}_{j+\frac12}\geq {Q}^{\text{\rm\tiny UP}}_{j+\frac12}=|a_{j+\frac12}|$.
\end{lemma}
\begin{proof}
If $ {  a_{j+\frac12}=\frac{f_{j+1}-f_j}{u_{j+1}-u_j}}\geq 0$,
then
\begin{align*}
\frac{f_j+f_{j+1}-2\hat{f}^{\text{\tiny E}}_{j+\frac12} }  {u_{j+1}-u_j}
 -\left|\frac{f_{j+1}-f_j}{u_{j+1}-u_j}\right|
=
 \frac{f_j+f_{j+1}-2\hat{f}^{\text{\tiny E}}_{j+\frac12} }  {u_{j+1}-u_j}
 -\frac{f_{j+1}-f_j}{u_{j+1}-u_j}
 =2\frac{f_j-\hat{f}^{\text{\tiny E}}_{j+\frac12} }  {u_{j+1}-u_j}\geq 0;
\end{align*}
If ${ a_{j+\frac12}=\frac{f_{j+1}-f_j}{u_{j+1}-u_j}}< 0$,
then
\begin{align*}
\frac{f_j+f_{j+1}-2\hat{f}^{\text{\tiny E}}_{j+\frac12} }  {u_{j+1}-u_j}
 -\left|\frac{f_{j+1}-f_j}{u_{j+1}-u_j}\right|
=
 \frac{f_j+f_{j+1}-2\hat{f}^{\text{\tiny E}}_{j+\frac12} }  {u_{j+1}-u_j}
 +\frac{f_{j+1}-f_j}{u_{j+1}-u_j}
 =2\frac{f_{j+1}-\hat{f}^{\text{\tiny E}}_{j+\frac12} }  {u_{j+1}-u_j}\geq 0.
\end{align*}
Hence, \ ${Q}^{\text{\tiny E}}_{j+\frac12}\geq {Q}^{\text{\tiny UP}}_{j+\frac12}=|a_{j+\frac12}|$.
\end{proof}
Lemma \ref{lemma3.1} implies that  three-point E scheme \eqref{fully-discrete-E-scheme01}
is at most first-order accurate  in space and time in the sense of local truncation error (LTE).

Moreover,  it holds the following conclusion.

 \begin{lemma}
 Under the CFL type condition 
\begin{align}\label{E-scheme-CFL001} \sigma {Q}^{\text{\tiny E},n}_{j+\frac12}\leq 1,
\end{align}
 the scheme \eqref{fully-discrete-E-scheme01}
 is TVD.
 \end{lemma}


\begin{remark}
Lemma 2.1 of \cite{Osher-SINUM1984} states: {\em The $2k+1$-point E scheme is at most first-order accurate.} Its proof is
{\small\em ``Following the proof that {monotone schemes} are at most first order accurate in [11],
we see that this weakened hypothesis also implies the same result for this explicit case. 
The semi-discrete case follows from a simple limiting procedure.''}
In fact, it is not true because of Lemma \ref{lemma2.2}.
\end{remark}

 Does it  hold generally that
\begin{align}\label{EQ:0321}
{{Q}^{\text{\tiny E}}_{j+\frac12}\geq   | f'(u)|}, 
\end{align}
for any $u$ between  $u_j$ and $u_{j+1}$, under the condition \eqref{eq:E flux001}? Generally, it is not  true,  here gives an example and a counterexample.
It is obvious that $Q^{\text{\tiny gLF}}_{j+\frac12}=q$ 
 satisfies \eqref{EQ:0321},
but   $Q^{\text{\tiny EO}}_{j+\frac12}$ may not because
$$
Q^{\text{\tiny EO}}_{j+\frac12}-|f'(u)|
=\frac1{u_{j+1}-u_j} \int^{u_{j+1}}_{u_j} |f'(s)|{\,\rm d}s-|f'(u)|
=\frac1{u_{j+1}-u_j}\int^{u_{j+1}}_{u_j} 
{ (|f'(s)|-|f'(u)|)}{\,\rm d}s,
$$
is not necessarily always non-negative.
%
%
%
%

\begin{lemma}
 Under the CFL type condition \eqref{E-scheme-CFL001} and
 \eqref{EQ:0321},
 the scheme \eqref{fully-discrete-E-scheme01} satisfies
\begin{align}\label{fully-discrete-E-scheme02}
\eta(u^{n+1}_j)
- \eta(u^{n}_{j}) 
+ {\sigma}  (\hat{q}^{\text{\tiny E},n}_{j+\frac12}-
\hat{q}^{\text{\tiny E},n}_{j-\frac12})
\leq 0,
\end{align}
where $\eta''(u)>0$, 
$\hat{q}^{\text{\tiny E}}_{j+\frac12}
=\frac12(q_j+q_{j+1})-\frac1{2\sigma}  {Q}^{\text{\tiny E}}_{j+\frac12} (\eta_{j+1}-\eta_j),
$ is consistent with $q(u)$
satisfying $q'(u)=\eta'(u)f'(u)$.
 \end{lemma}
 
\begin{remark}The condition \eqref{EQ:0321} is a bit too strong. It is required in the following proof.
 In fact, the conservative three-point monotone scheme is the E scheme and satisfies the entropy inequality for any convex entropy pair with no requirement of the condition \eqref{EQ:0321}, see \cite{HHL1976}. 
\end{remark}
\begin{proof}
 If denoting the left hand side of  \eqref{fully-discrete-E-scheme02} by ${ I}$,
then one has to prove { $I\leq 0$.}

Multiplying \eqref{fully-discrete-E-scheme01} by { $\eta'(u_j^{n+1})$}
gets
{ \begin{align*}
0=&\eta'(u_j^{n+1}) (u^{n+1}_j-u^{n}_j)
-\frac{\eta'(u_j^{n+1})}{2} {Q}^{\text{\tiny E},n}_{j+\frac12} (u_{j+1}^n-u_{j}^n)
+\frac{\sigma}{2}\eta'(u_j^{n+1}) (f(u^{n}_{j+1})-f(u^{n}_{j}))\\
&+\frac{\eta'(u_j^{n+1})}{2}   {Q}^{\text{\tiny E},n}_{j-\frac12} (u_{j}^n-u_{j-1}^n)
+\frac{\sigma}{2}\eta'(u_j^{n+1}) (f(u^{n}_j)-f(u^{n}_{j-1}))
\\
& =-\frac{\sigma}2\int_{u^{n}_j}^{u^{n}_{j+1}} \eta'(u_j^{n+1})
({Q}^{\text{\tiny E},n}_{j+\frac12}  - f'(s) )\ {\rm d}s
  -\frac{\sigma}2\int_{u^{n}_j}^{u^{n}_{j-1}} \eta'(u_j^{n+1})
({Q}^{\text{\tiny E},n}_{j-\frac12}  +  f'(s) )\ {\rm d}s
\\ & 
+\int^{u_j^{n+1}}_{u^{n}_j}\eta'(u_j^{n+1})\ {\rm d}s
=:{ II}.
\end{align*}}%
{ Because$(II-I)+I=II=0$,
 one has $I=-(II-I)\leq 0$ if    $II-I\geq 0$.}

Note that  we do not use the following identity here
$$f(u^{n}_{j+1})-f(u^{n}_{j})=a^n_{j+\frac12}(u_{j+1}^n-u_{j}^n)
=\int_{u^{n}_j}^{u^{n}_{j+1}}   a^n_{j+\frac12}  \ {\rm d}s.
$$
Because
{\small\begin{align*}
I=& \eta(u^{n+1}_j) -\eta(u_j^n) +
\frac{\sigma} 2(q_{j+1}^n- q_j^n+q_j^n-q_{j-1}^n)
-\frac1{2}  {Q}^{\text{\tiny E},n}_{j+\frac12} (\eta_{j+1}-\eta_j)
+\frac1{2}  {Q}^{\text{\tiny E},n}_{j-\frac12} (\eta_{j}-\eta_{j-1})
\\
=&-\frac{\sigma} 2\int_{u^{n}_j}^{u^{n}_{j+1}} \eta'(s)
({Q}^{\text{\tiny E},n}_{j+\frac12}  - f'(s))\ {\rm d}s
-\frac{\sigma}  2\int_{u^{n}_j}^{u^{n}_{j-1}} \eta'(s)
({Q}^{\text{\tiny E},n}_{j-\frac12}  +  f'(s))\ {\rm d}s
+\int^{u_j^{n+1}}_{u^{n}_j}\eta'(s)\ {\rm d}s,
\end{align*}}%
one has (using the assumption that { ${Q}^{\text{\tiny E},n}_{j+\frac12} \geq |f'(u)|$ for any $u$ between $u_j$ and $u_{j+1}$})
{ \begin{align*}
II-I=&
\frac{\sigma}2\int_{{ u^{n}_j}}^{u^{n}_{j+1}} ({Q}^{\text{\tiny E},n}_{j+\frac12}  -   f'(s))(\eta'(s)-\eta'(u_j^{n+1}))
\ {\rm d}s
+\frac{\sigma}2\int_{{ u^{n}_j}}^{u^{n}_{j-1}} ({Q}^{\text{\tiny E},n}_{j-\frac12}  +   f'(s))(\eta(s)-\eta'(u_j^{n+1})
\ {\rm d}s
\\&
-\int^{u_j^{n+1}}_{u^{n}_j}(\eta'(s)-\eta'(u_j^{n+1}))\ {\rm d}s
= 
\frac{\sigma}2\Big(\int_{{ u^{n+1}_j}}^{u^{n}_{j+1}}
+\int^{{ u^{n+1}_j}}_{u^{n}_{j}} \Big)
({Q}^{\text{\tiny E},n}_{j+\frac12}   -  f'(s))(\eta'(s)-\eta'(u_j^{n+1}))
\ {\rm d}s
\\
&+\frac{\sigma} 2\Big(\int_{{ u^{n+1}_j}}^{u^{n}_{j-1}}
+\int^{{ u^{n+1}_j}}_{u^{n}_{j}} \Big)
({Q}^{\text{\tiny E},n}_{j-\frac12}   +  f'(s))(\eta(s)-\eta'(u_j^{n+1})
\ {\rm d}s
-\int^{u_j^{n+1}}_{u^{n}_j}(\eta'(s)-\eta'(u_j^{n+1}))\ {\rm d}s
\\= &
\frac{\sigma}2\int_{{ u^{n+1}_j}}^{u^{n}_{j+1}}
({Q}^{\text{\tiny E},n}_{j+\frac12}   -  f'(s))(\eta'(s)-\eta'(u_j^{n+1}))
\ {\rm d}s
+\frac{\sigma}2 \int_{{ u^{n+1}_j}}^{u^{n}_{j-1}}
({Q}^{\text{\tiny E},n}_{j-\frac12}   +   f'(s))(\eta(s)-\eta'(u_j^{n+1})
\ {\rm d}s
\geq 0.
\end{align*}}%
Note that the first term on the right side of the first equal sign is less than 0 (caused by the explicit Euler time
discretization), while the other two terms are greater than 0.
\end{proof}

\begin{remark} If $\eta=\frac12 u^2$, then
the aforementioned entropy inequality is a common energy inequality.
\end{remark} 

%

If applying  the { implicit Euler time discretiztion} to the semi-discrete E scheme \eqref{eq:scheme001}-\eqref{eq:E flux001}, 
then one has the fully discrete  (implicit) E scheme
\begin{align}\label{fully-discrete-E-scheme01b}
u_j^{n+1}=u_j^n-\sigma(\hat{f}^{\text{\tiny E},n+1}_{j+\frac12}-\hat{f}^{\text{\tiny E},n+1}_{j-\frac12}),\ \  \sigma=\frac{\tau_n}{h}.
\end{align}
The numerical flux may be rewritten as 
$$
\hat{f}^{\text{\tiny E}}_{j+\frac12}
=\frac12(f_j+f_{j+1})-\frac1{2\sigma}  {Q}^{\text{\tiny E}}_{j+\frac12} (u_{j+1}-u_j),
$$
where
$$
\frac1{\sigma}{Q}^{\text{\tiny E}}_{j+\frac12}({u_{j+1}-u_j})
=  (f_j+f_{j+1}-2\hat{f}^{\text{\tiny E}}_{j+\frac12}),
$$
or
$$
{Q}^{\text{\tiny E}}_{j+\frac12}=
{\sigma}   \frac{f_j+f_{j+1}-2\hat{f}^{\text{\tiny E}}_{j+\frac12} }  {u_{j+1}-u_j}.
$$
 The condition \eqref{eq:E flux001} for the E scheme implies 
 the viscous coefficient is non-negative, i.e.
 ${Q}^{\text{\tiny E}}_{j+\frac12}\geq 0$.

\begin{lemma}
 The scheme \eqref{fully-discrete-E-scheme01b} satisfies
\begin{align}\label{fully-discrete-E-scheme02b}
\eta(u^{n+1}_j)
- \eta(u^{n}_{j}) 
+ {\sigma}  (\hat{q}^{\text{\tiny E},n+1}_{j+\frac12}-
\hat{q}^{\text{\tiny E},n+1}_{j-\frac12})
\leq 0,
\end{align}
where $\eta''(u)>0$, 
$\hat{q}^{\text{\tiny E}}_{j+\frac12}$ is consistent with $q(u)$
satisfying $q'(u)=\eta'(u)f'(u)$.
 \end{lemma}
 
\begin{remark} 
The fully discrete E scheme with other non-negatively dissipative time discretization satisfies the entropy inequality for any convex entropy pair unconditionally.
 \end{remark} 
  
\begin{proof}  
Multiplying \eqref{fully-discrete-E-scheme01b} by { $\eta'(u_j^{n+1})$}
gets
{\small\begin{align*}
0=&\eta'(u_j^{n+1}) (u^{n+1}_j-u^{n}_j)
+\sigma \eta'(u_j^{n+1})  \left[\hat{f}^{\text{\tiny E},n+1}_{j+\frac12}
-\hat{f}^{\text{\tiny E},n+1}_{j-\frac12}\right]
\\ 
=&{ \eta(u_j^{n+1})  -\eta(u_j^{n}) }
+\big[ \eta'(u_j^{n+1}) (u^{n+1}_j-u^{n}_j) {  -\eta(u_j^{n+1})  +\eta(u_j^{n}) }\big]
+\sigma \eta'(u_j^{n+1})  \left[\hat{f}^{\text{\tiny E},n+1}_{j+\frac12}
-\hat{f}^{\text{\tiny E},n+1}_{j-\frac12}\right].
\end{align*}}%
At the above RHS,  the second term  
is non-negative, that is, 
\begin{align*}
\eta'(u_j^{n+1}) (u^{n+1}_j-u^{n}_j) {  -\eta(u_j^{n+1})  +\eta(u_j^{n}) }=  \int_{u^n_j}^{u^{n+1}_j}  \left[\eta'(u_j^{n+1})  -\eta'(s)\right] {\,\rm d}s\geq 0, \ \
 \end{align*} 
while the third term  may be transformed into the divergence form, see the proof of Theorem \ref{theorem001}, thus
 {\small\begin{align*}
 { \eta(u_j^{n+1})  -\eta(u_j^{n}) }
 +\sigma   \left[\hat{q}^{\text{\tiny E},n+1}_{j+\frac12}
-\hat{q}^{\text{\tiny E},n+1}_{j-\frac12}\right]\leq 0.
\end{align*}}%
\end{proof}

 \section{Sweby flux-limiter scheme}\label{section04}
 
Assume $a>0$,  rewrite the Lax-Wendroff (LW) scheme of $u_t+a u_x=0$ as
\begin{equation}\label{AA0}
u^{n+1}_j=u^n_j-\nu(u^n_j-u^n_{j-1}){ -\frac12\nu(1-\nu)(u^n_{j+1}-2u^n_j+u^n_{j-1})},
\end{equation}
The idea of the Sweby flux-limiter scheme is 
to limit { the anti-diffusive term} or { the anti-diffusive  flux}, 
that is,
to modify \eqref{AA0} as the following flux-limiter scheme \cite{Sweby-SINUM1984}
\begin{equation}\label{f4}
u^{n+1}_j=u^n_j-\nu \Delta_x^- u_j^n-\Delta_x^- ({ \varphi_j^n} \frac12\nu (1-\nu)
\Delta^+_x u_j^n),
\end{equation}
where $\varphi_j$ is some limiter, 
its value is non-negative, so that  the coefficient sign of the anti-diffusive term in \eqref{AA0} may be preserved.
The numerical flux of the flux-limier scheme \eqref{f4} is $$\hat{f}^{\text{\tiny FL}}_{j+\frac12}= au_j^n+{ \varphi_j^n} \frac12a (1-\nu)
(u_{j+1}^n- u_j^n).
$$

When Sweby extended the flux-limiter scheme (\ref{f4}) 
to the quasilinear conservation laws, he used the { E flux-based splitting}
  \begin{align*} 
\nu^n_{j+\frac12}= \sigma a^n_{j+\frac12}
= \sigma \frac{f^n_{j+1}-f^n_j} {u^n_{j+1}-u^n_j}
 =  \sigma \frac{   f_{j+1}^n-\hat{f}_{j+\frac12}^{\text{\tiny E},n} }  {\Delta u^n_{j+\frac12}}
 +  \sigma \frac{ \hat{f}_{j+\frac12}^{\text{\tiny E},n}-f^n_j}  {\Delta u^n_{j+\frac12}}
 \\
 =  \sigma \frac{(\Delta f_{j+\frac12})^{n,+}}  {\Delta u^n_{j+\frac12}} + \sigma \frac{(\Delta f_{j+\frac12})^{n,-}}  {\Delta u^n_{j+\frac12}}
 =:\nu^{n,+}_{j+\frac12}+\nu^{n,-}_{j+\frac12},
 \end{align*} 
 where $
 (\Delta f_{j+\frac12})^{n,+}:=f_{j+1}^n- \hat{f}^{\text{\tiny E},n}_{j+\frac12}$, 
$(\Delta f_{j+\frac12})^{n,-}:= \hat{f}_{j+\frac12}^{\text{\tiny E},n}-f_{j}^n$,
$\Delta u^n_{j+\frac12}=u^n_{j+1}-u^n_j$,
 $ \nu^{n,+}_{j+\frac12}\geq 0$, $\nu^{n,-}_{j+\frac12}\leq 0$.
 Define
 \begin{align*}
& \alpha^{n,\pm}_{j+\frac12} :=\frac12 \left(1\mp \nu^{n,\pm}_{j+\frac12}\right), \ \
 r^{n,+}_j:=\frac{\alpha^{n,+}_{j-\frac12}(\Delta f_{j-\frac12})^{n,+}}
{\alpha^{n,+}_{j+\frac12}(\Delta f_{j+\frac12})^{n,+}}, \ \
 r^{n,-}_j:=\frac{\alpha^{n,-}_{j+\frac12}(\Delta f_{j+\frac12})^{n,-}}
{\alpha^{n,-}_{j-\frac12}(\Delta f_{j-\frac12})^{n,-}}.
 \end{align*} 
Thus,  the LW scheme may be rewritten as ($u_{tt}=-\partial_x(f_{x})=\partial_x(f' f_x) $)
 \begin{align*}
 &u^{n+1}_j=u^n_j-\frac{\sigma}2  (  f_{j+1}^n -f_{j-1}^n )
 + \frac{\sigma^2 a^n_{j+\frac12} }2  ( f_{j+1}^n -f_{j}^n)
- \frac{ \sigma^2 a^n_{j-\frac12}  }2  ( f_{j}^n -f_{j-1}^n)\\
&
= u^n_j- {\sigma} ({ \hat{f}_{j+\frac12}^{\text{\tiny E},n} -\hat{f}_{j-\frac12}^{\text{\tiny E},n}})
+\frac{\sigma}2   \Delta_x^-  \left[ { \hat{f}_{j+\frac12}^{\text{\tiny E},n}}  -f_{j}^n
+{ \hat{f}_{j+\frac12}^{\text{\tiny E},n}  } -f_{j+1}^n   
\right]   
 \\ &+\frac{\sigma}2   \Delta_x^- \Big[\left(
 \nu^{n,+}_{j+\frac12}+\nu^{n,-}_{j+\frac12}
 \right) ( f_{j+1}^n { -\hat{f}_{j+\frac12}^{\text{\tiny E},n}+\hat{f}_{j+\frac12}^{\text{\tiny E},n} } -f_{j}^n)\Big]
 \\ &= 
 u^n_j- {\sigma} (\hat{f}_{j+\frac12}^{\text{\tiny E},n} -\hat{f}_{j-\frac12}^{\text{\tiny E},n})
+\frac{\sigma}2   \Delta_x^- \Big[
(\Delta f_{j+\frac12})^{n,-}-(\Delta f_{j+\frac12})^{n,+}
\\ &+ { 
( \nu^{n,+}_{j+\frac12}+\nu^{n,-}_{j+\frac12} ) [ (\Delta f_{j+\frac12})^{n,+}+(\Delta f_{j+\frac12})^{n,-}]   }
 \Big] .
 \end{align*} 
If ${   ( \nu^{n,+}_{j+\frac12}+\nu^{n,-}_{j+\frac12} ) [ (\Delta f_{j+\frac12})^{n,+}+(\Delta f_{j+\frac12})^{n,-}]   }
{ \approx}
  \nu^{n,+}_{j+\frac12}  (\Delta f_{j+\frac12})^{n,+}+
  \nu^{n,-}_{j+\frac12} (\Delta f_{j+\frac12})^{n,-}
  $ and use the limiter, then 
the { Sweby flux-limiter scheme} becomes
\begin{align}
u^{n+1}_j=&u^n_j-\sigma (\hat{f}_{j+\frac12}^{\text{\tiny E},n}-\hat{f}^{\text{\tiny E},n}_{j-\frac12})
-\sigma \Delta_x^- \Big\{{  \varphi (r^{n,+}_j)}
\alpha^{n,+}_{j+\frac12} (\Delta f_{j+\frac12})^{n,+}
-{ \varphi (r^{n,-}_{j+1})}\alpha^{n,-}_{j+\frac12}
(\Delta f_{j+\frac12})^{n,-} \Big\},
 \label{EQ-chap05-Sweby001}
\end{align}

\begin{lemma}
The Sweby flux-limiter scheme \eqref{EQ-chap05-Sweby001} is TVD under a 
 CFL type condition \cite{Sweby-SINUM1984}
$$
\sup_{\xi}\{\sigma |f'(\xi)|\} \leq \left(\frac2{2+\Phi}\right)\mu\leq \frac23.
$$
\end{lemma}

With the help of the E flux based splitting,  the LW scheme may also be rewritten  in incremental form
 \begin{align*}
 u^{n+1}_j=&u^n_j-\frac{\sigma}2  [ {  f_{j+1}^n -f_{j-1}^n} ]
 + \frac{(\sigma a^n_{j+\frac12})^2 }2  ( u_{j+1}^n -u_{j}^n)
- \frac{(\sigma a^n_{j-\frac12})^2 }2  ( u_{j}^n -u_{j-1}^n)\\
=& u^n_j-\frac{\sigma}2(f_{j+1}^n
{ -\hat{f}_{j+\frac12}^{\text{\tiny E}}+\hat{f}_{j+\frac12}^{\text{\tiny E}}}
 { -f_j^n+f_j^n}
 { -\hat{f}_{j-\frac12}^{\text{\tiny E}}+\hat{f}_{j-\frac12}^{\text{\tiny E}}}
 - f_{j-1}^n)
 \\ &+\frac{1}2  \left(
 \nu^{n,+}_{j+\frac12}+\nu^{n,-}_{j+\frac12}
 \right)^2( u_{j+1}^n -u_{j}^n)
  -\frac{1}2  \left(
  \nu^{n,+}_{j-\frac12}+\nu^{n,-}_{j-\frac12}
  \right)^2( u_{j}^n -u_{j-1}^n)
 \\ =&
 u^n_j
 -\frac12\left [ { 
 (\nu^{n,+}_{j+\frac12}+\nu^{n,-}_{j+\frac12}) \Delta u^n_{j+\frac12} 
 -(\nu^{n,+}_{j-\frac12}+\nu^{n,-}_{j-\frac12}) \Delta u^n_{j-\frac12} } \right]
\\ & +\frac{1}2  \left(
 \nu^{n,+}_{j+\frac12}+\nu^{n,-}_{j+\frac12}
 \right)^2  \Delta u^n_{j+\frac12} 
  -\frac{1}2  \left(
  \nu^{n,+}_{j-\frac12}+\nu^{n,-}_{j-\frac12}
  \right)^2 \Delta u^n_{j-\frac12} 
  .
 \end{align*} 

Let us  discuss the accuracy of the scheme \eqref{EQ-chap05-Sweby001}
in the sense of LTE. 
 Let $\varphi_j\equiv 1$, then \eqref{EQ-chap05-Sweby001} may 
 be rewritten as 
 \begin{align}\nonumber
 u_j^{n+1}=&u_j-\frac\sigma 2(f_{j+1}-f_{j-1})\\ \nonumber
 &+\frac{\sigma^2}{2}\Big [{ 
 \frac{ f_{j+1}-\hat{f} ^{\text{\tiny E},n}_{j+\frac12}}{u_{j+1}-u_j}
 ( f_{j+1}-\hat{f}^{\text{\tiny E},n} _{j+\frac12})
 +\frac{\hat{f} ^{\text{\tiny E},n}_{j+\frac12}-f_j}{u_{j+1}-u_j}
 ( \hat{f} _{j+\frac12}-f_j) }\Big]
 \\  \label{EQ-chap05-Sweby002}
 &-
 \frac{\sigma^2}{2}\Big [
 \frac{ f_{j}-\hat{f}^{\text{\tiny E},n} _{j-\frac12}}{u_{j}-u_{j-1}}
 ( f_{j}-\hat{f} ^{\text{\tiny E},n}_{j-\frac12})
 +\frac{\hat{f} ^{\text{\tiny E},n}_{j-\frac12}-f_{j-1}}{u_{j}-u_{j-1}}( \hat{f} _{j-\frac12}-f_{j-1})
 \Big].
 \end{align}
Because $(\Delta f_{j+\frac12})^{+}+(\Delta f_{j+\frac12})^{-}=f_{j+1}-f_j=a_{j+\frac12}(u_{j+1}-u_j)$,
\begin{align*}
  a_{j+\frac12}^2 (u_{j+1}-u_j)
=&
{\frac{ f_{j+1}-\hat{f} ^{\text{\tiny E},n}_{j+\frac12}}{u_{j+1}-u_j}
 ( f_{j+1}-\hat{f} ^{\text{\tiny E},n}_{j+\frac12})
 +\frac{\hat{f} ^{\text{\tiny E},n}_{j+\frac12}-f_j}{u_{j+1}-u_j}
 ( \hat{f} ^{\text{\tiny E},n}_{j+\frac12}-f_j)      }
\\ & {+2 \frac{(\Delta f_{j+\frac12})^{+}(\Delta f_{j+\frac12})^{-}}{u_{j+1}-u_j}}.
  \end{align*}
Comparing it with the viscous term in \eqref{EQ-chap05-Sweby002}, 
   the viscous term in \eqref{EQ-chap05-Sweby002} is missing the last term on the right-hand side of the above equation.
%
 It can be inferred that 
 { \eqref{EQ-chap05-Sweby002} may not possess second-order accuracy in both space and time, and  so  is the Sweby flux-limiter scheme \eqref{EQ-chap05-Sweby001}}.


  Does there exist $\hat{f}_{j+\frac12}$ such that
 $$ \left(
 \frac{f_{j+1}-f_j} {u_{j+1}-u_j}   \right)^2
 = \left(   \frac{   f_{j+1}-\hat{f}_{j+\frac12} }  {u_{j+1}-u_j}  \right)^2
 +\left(   \frac{ \hat{f}_{j+\frac12}-f_j}  {u_{j+1}-u_j}   \right)^2?
 $$
If $f'(u)$ is always nonnegative and takes $\hat{f}_{j+\frac12}=f_j$,
or $f'(u)$is always nonpositive and takes$\hat{f}_{j+\frac12}=f_{j+1}$, 
 then the above identity holds,
so the scheme \eqref{EQ-chap05-Sweby001}
has second-order accuracy in the sense of local truncation error and is TVD.
 %
  If $\hat{f}_{j+\frac12}=\frac12(f_j+f_{j+1})-\frac12|a_{j+\frac12}|(u_{j+1}-u_j)$, { generally, it does not satisfy the E flux condition \eqref{eq:E flux001}},
  where $a_{j+\frac12}(u_{j+1}-u_j)=f_{j+1}-f_j$,  
  then the above identity holds.

 If the numerical flux is that of the three-point upwind scheme, that is,
 $$\hat{f}_{j+\frac12} 
 =\frac12(f_j+f_{j+1})-\frac{|a_{j+\frac12} |} 2(u_{j+1}-u_j),\
 a_{j+\frac12}(u_{j+1}-u_j)=f_{j+1}-f_{j},
 $$
 then
  \begin{align*}
 (\Delta f_{j+\frac12})^{+} &=f_{j+1}-\hat{f}_{j+\frac12} 
  =\frac12(f_{j+1}-f_j)+\frac{|a_{j+\frac12} |} 2(u_{j+1}-u_j)
  \\ &
  =\frac{a_{j+\frac12}+|a_{j+\frac12} |} 2(u_{j+1}-u_j),
  \\
  (\Delta f_{j+\frac12})^{-}
&= \hat{f}_{j+\frac12} -f_j
=\frac12(f_{j+1}-f_j)-\frac{|a_{j+\frac12} |} 2(u_{j+1}-u_j)
\\ & =\frac{a_{j+\frac12}-|a_{j+\frac12} |} 2(u_{j+1}-u_j).
    \end{align*}
Hence, 
 \begin{align*}
 \nu_{j+\frac12}^+=&\sigma\frac{(\Delta f_{j+\frac12})^{+}}{\Delta u_{j+\frac12}^n}
 =\sigma\frac{a_{j+\frac12}+|a_{j+\frac12} |} 2\geq 0,
 \\
  \nu_{j+\frac12}^-=&\sigma\frac{(\Delta f_{j+\frac12})^{-}}{\Delta u_{j+\frac12}^n}
 =\sigma\frac{a_{j+\frac12}-|a_{j+\frac12} |} 2\leq 0,
 \\
{(\Delta f_{j+\frac12})^{+}(\Delta f_{j+\frac12})^{-}}
=&\frac14  (a_{j+\frac12}+|a_{j+\frac12} |)  
 (a_{j+\frac12}-|a_{j+\frac12} |) (u_{j+1}-u_j)^2 
 \\
 & 
 =\frac14(a_{j+\frac12}^2-|a_{j+\frac12} |^2)   
 (u_{j+1}-u_j) ^2={0}.
  \end{align*}
It means that 
the scheme \eqref{EQ-chap05-Sweby001}
is second-order accurate in the sense of LTE and TVD.
 
 If the numerical flux is that of the EO scheme (satisfying the E flux condition \eqref{eq:E flux001}), 
that is,
 $$\hat{f}_{j+\frac12} 
 =\frac12(f_j+f_{j+1})-\frac12\int_{u_j}^{u_{j+1}} |a(\xi)|~{\rm d}\xi,
 $$
 then
 \begin{align*}
& { (\Delta f_{j+\frac12})^{+}(\Delta f_{j+\frac12})^{-}}
=\frac14 \int_{u_j}^{u_{j+1}}  (a(\xi)+ |a(\xi)|)~{\rm d}\xi
 \int_{u_j}^{u_{j+1}}  (a(\xi)- |a(\xi)|)~{\rm d}\xi,
  \end{align*}
 is { not necessarily zero}. 
For example, if $u_j=0$, $u_{j+1}=2$,
 $$a(u)=\begin{cases} 1, & u\in [0,1],\\
 -1, & u\in (1,2],\end{cases}
 $$
then
  \begin{align*}
&  \int_{u_j}^{u_{j+1}}  (a(\xi)- |a(\xi)|)~{\rm d}\xi
=\int_0^1(1-1)~{\rm d}\xi+\int_1^2(-1-1)~{\rm d}\xi=-2,\\
\\
& \int_{u_j}^{u_{j+1}}  (a(\xi)+ |a(\xi)|)~{\rm d}\xi
=\int_0^1(1+1)~{\rm d}\xi+\int_1^2(-1+1)~{\rm d}\xi=2,
\end{align*}
so that
 \begin{align*}
& { (\Delta f_{j+\frac12})^{+}(\Delta f_{j+\frac12})^{-}}
=\frac14 \times(-2)\times 2=-1\neq 0.
  \end{align*}

 Is it possible to add corresponding terms to the anti-diffusion term in \eqref{EQ-chap05-Sweby001}, so that the scheme can ensure both second-order accuracy in space and time and TVD property?
  Unfortunately,  there is no definite answer at the moment.

\section{Conclusion}\label{section05}

This paper has revisited Osher's E scheme    \cite{Osher-SINUM1984}  and  Sweby's flux-limiter scheme for quasilinear hyperbolic conservation laws  \cite{Sweby-SINUM1984}.
For 1D scalar conservation laws, some existing and newly presented results are listed as follows:  
  \begin{itemize}
  \item 
   { Two-point monotone flux is E flux, but conversely it may not necessarily
be correct,}
%
    including  {  the} generalized LF flux and {  the} EO flux { etc.}
  If $u_{j+1}\geq u_j$ { (resp. $u_{j+1}< u_j$)}, { the} Godunov flux is the upper { (resp. lower)} boundary of  the E flux. 
  \item Multi-point { (three or more points) monotone} flux may not necessarily be { E} flux, and multi-point { E} flux may not necessarily be { monotone} flux.
  \item Semi-discrete E scheme satisfies the entropy inequality for any
  convex entropy pair, but numerical entropy flux is not unique.
  
  \item For the fully discrete (three-point) E scheme with explicit Euler,
its viscous coefficient  is not less than 0 and 
 $ |a_{j+\frac12}|$,
  so that it is at most first-order accurate  in the sense of LTE.
 Moreover,   
under the CFL type condition $\sigma \hat{Q}^{\text{\tiny E}}_{j+\frac12}\leq 1$, 
it is TVD, and  satisfies  the entropy inequality for any
  convex entropy pair when  
  $\hat{Q}^{\text{\tiny E}}_{j+\frac12}\geq \max_{u\in [u_j,u_{j+1}} \{|f'(u)|\}$  { which is a bit too severe  and generally does not hold}.
It is known that the conservative three-point monotone scheme (E scheme) satisfies the entropy inequality for any convex entropy pair without such severe condition.
 
\item The fully discrete E scheme with 
or non-negatively dissipative time discretization (e.g. { the} implicit Euler) satisfies  the entropy inequality for any
  convex entropy pair unconditionally.

\item Sweby's  flux-limiter scheme for the quasilinear conservation laws was built on 
the E flux-based splitting  $f_{j+1}-f_j=f_{j+1} { -\hat{f}^{\text{\tiny E}}_{j+\frac12}+\hat{f}^{\text{\tiny E}}_{j+\frac12}}-f_j$ and the LW scheme. It may not be second-order accurate in both space and time.

\end{itemize}
 
	\section*{Acknowledgments}
The work was partially supported by  the National Natural Science Foundation of China (No.12288101). 
The author would also like to thank Professor Junming Duan of The Chinese University of Hong Kong, Shenzhen and Mr. Tengfei Zheng 
for numerous discussions
during the preparation of this work.


\end{document}